\documentclass[a4paper,12pt]{amsart}

\usepackage{amssymb}
\usepackage{hyperref}
\usepackage[normalem]{ulem}

\def\altdb{\vadjust{\vbox to 0pt{\vss\hbox{\kern \hsize
\quad{\dbend}}\kern\baselineskip\kern-10pt}}}

\newcommand\field[1]{\mathbb{#1}}
\newcommand\CC{\field{C}}
\newcommand\FF{\field{F}}
\newcommand\NN{\field{N}}
\newcommand\QQ{\field{Q}}
\newcommand\RR{\field{R}}
\newcommand\TT{\field{T}}
\newcommand\ZZ{\field{Z}}

\newcommand\Bb{\mathcal B}
\newcommand\Ff{\mathcal F}
\newcommand\Gg{G}

\newcommand\Uu{\mathcal U}

\newcommand\lsp{\operatorname{span}}
\newcommand\clsp{\operatorname{\overline{span\!}\,\,}}
\newcommand\supp{\operatorname{supp}}
\newcommand\Ind{\operatorname{Ind}}

\newcommand{\go}{\Gg^{(0)}}
\newcommand{\Bco}[2]{B^{\operatorname{co}}_{#1}(#2)}
\newcommand{\BcoG}{\Bco{*}{\Gg}}

\newcommand{\ag}{A(\Gg)}

\newcommand{\inv}{^{-1}}
\newcommand{\unit}{^{(0)}}
\newcommand{\cs}{\ensuremath{C^{*}}}

\newcommand{\Iso}{\operatorname{Iso}}
\newcommand{\End}{\operatorname{End}}

\theoremstyle{plain}
\newtheorem{theorem}{Theorem}[section]
\newtheorem*{theorem*}{Theorem}
\newtheorem*{prop*}{Proposition}
\newtheorem{cor}[theorem]{Corollary}
\newtheorem{lemma}[theorem]{Lemma}
\newtheorem{prop}[theorem]{Proposition}

\theoremstyle{remark}
\newtheorem{rmk}[theorem]{Remark}

\newtheorem{example}[theorem]{Example}
\newtheorem{examples}[theorem]{Examples}
\newtheorem{claim}[theorem]{Claim}

\theoremstyle{definition}
\newtheorem{dfn}[theorem]{Definition}

\usepackage{color}

\title{Simplicity of algebras associated to \'etale groupoids}

\author{Jonathan Brown}
\address{Jonathan Brown\\
Mathematics Department\\
    Kansas State University\\
138 Cardwell Hall\\
Manhattan\\
Kansas 66506-2602\\
USA}
\email{brownjh@math.ksu.edu}

\author{Lisa Orloff Clark}
\address{Lisa Orloff Clark\\
    Department of Mathematics and Statistics\\
    University of Otago\\
    PO Box 56\\
    Dunedin 9054\\
    New Zealand}
\email{lclark@maths.otago.ac.nz}

\author{Cynthia Farthing}
\address{Cynthia Farthing\\
    Department of Mathematics\\
    University of Iowa\\
    14 MacLean Hall\\
    Iowa City\\
    Iowa 52242-1419\\
    USA}
\email{cynthia-farthing@uiowa.edu}

\author{Aidan Sims}
\address{Aidan  Sims\\
    School of Mathematics and Applied Statistics\\
    University of Wollongong\\
    NSW 2522\\
    Australia}
\email{asims@uow.edu.au}

\date{December 4, 2012}

\thanks{This research has been supported by the Australian Research Council.}

\begin{document}

\begin{abstract}
We prove that the full $C^*$-algebra of a second-countable, Hausdorff, \'etale, amenable
group\-oid is simple if and only if the groupoid is both topologically principal and
minimal. We also show that if $G$ has totally disconnected unit space, then the
complex $^*$-algebra of its inverse semigroup of compact open bisections, as introduced by Steinberg, is simple if and only if $G$ is both effective and minimal.
\end{abstract}

\maketitle


\section{Introduction}

Let $G$ be a groupoid which is \'etale in the sense that $r,s : G \to \go$ are local
homeomorphisms.   Complex algebras $A(G)$ associated to locally compact, Hausdorff,
\'etale groupoids $G$ with totally disconnected unit spaces were introduced in
\cite{Steinberg2010}.  There, Steinberg shows that $A(G)$ can be used to describe
inverse-semigroup algebras. These algebras, which we call Steinberg algebras, were also
examined in \cite{CFST} where they are shown to include the complex Kumjian-Pask algebras
of higher-rank graphs \cite{ArandaClarkEtAl:xx11}, and hence the complex Leavitt path
algebras of directed graphs \cite{AbrPino}.   In general, $A(G)$ is dense in $\cs(G)$,
the $\cs$-algebra associated to $G$. The criteria of \cite{RS07} which characterise
simplicity of a higher-rank graph $C^*$-algebra also characterise simplicity of the
associated Kumjian-Pask algebra \cite[Theorem~5.14]{ArandaClarkEtAl:xx11}. Encouraged by
this, we set out to investigate the simplicity of $A(G)$.

Translating from the higher-rank graph setting, we hoped to prove that $G$ is
\emph{topologically principal}\footnote{See Remark~\ref{rmk:terminology} regarding
terminology.} in the sense that the units with trivial isotropy are dense in the unit
space, and \emph{minimal} in the sense that the unit space has no nontrivial open
invariant subsets, if and only if $A(G)$ is simple.  Although the ``if" implication was
not known in the $\cs$-algebra setting, we hoped that in the situation of algebras, where
there are no continuity hypotheses to check when constructing representations, we could
adapt the ideas of \cite[Proposition~3.5]{RS07}. Our initial attempts to prove the result
failed.  We eventually realised that the natural necessary condition is not that $G$ be
topologically principal, instead it is that $G$ be \emph{effective}: every open subset of $G
\setminus \go$ contains an element $\gamma$ such that $r(\gamma)$ and $s(\gamma)$ are distinct. 
For if an \'etale groupoid $G$ with totally
disconnected unit space is not effective, then there exists a compact open set $B
\subseteq G \setminus \go$ consisting purely of isotropy on which the range and source
maps are homeomorphisms. It follows that $1_{r(B)} - 1_B$ belongs to $A(G)$ and vanishes
under a natural homomorphism from $A(G)$ to the algebra of endomorphisms of the free
complex module $\FF(\go)$ with basis $\go$ (see
Proposition~\ref{prp:pi_inj_implies_3.1}). That $G$ is effective is, in general, a
strictly weaker condition than that it is topologically principal (see Examples
\ref{ex:2ndctble ex}~and~\ref{ex:etale ex}), though they are equivalent in the
higher-rank graph setting.  We show that effectiveness, together with minimality, is
necessary and sufficient for simplicity of $A(G)$ (see Theorem~\ref{thm:alg_simple}).

It came as a surprise to discover that the arguments we had developed for $A(G)$ could be
adapted to give new results in the $C^*$-algebraic setting provided that $G$ is
second-countable; this amounts to restricting our attention to separable $\cs$-algebras.
In this setting we can also drop the requirement that $\go$ is totally disconnected. A
Baire-category argument \cite[Proposition~3.6]{Renault:IMSB08} shows that a
second-countable, Hausdorff and \'etale groupoid $G$ is effective if and only if it is
topologically principal. Combining all of this, we fill in the missing piece of the
simplicity puzzle for \'etale groupoid $\cs$-algebras. That is, we show that if $\cs(G)$
is simple, then $G$ must be topologically principal. Hence we are able to give necessary
and sufficient conditions for the simplicity of $\cs(G)$ as well. Though some parts of
what we have done can be found in the literature, we have taken pains to make our results
self-contained and to take the most elementary path possible. There are many classes of
$C^*$-algebras with \'etale groupoid models (see for example \cite{Dea95, ER07, FMY,
GiordanoPutnamEtAl:Crelle95, IonescuMuhly, KPR, QuiggSieben:JAMSSA99, Ren80,
Renault:IMSB08, Y1}), so we expect that our results will find numerous applications.

After a short preliminaries section, we describe in Section~\ref{sec:groupoid} a number
of equivalent conditions to a locally compact, Hausdorff, \'etale groupoid $G$ being
effective. We show that these equivalent conditions are formally weaker than $G$ being
topologically principal, but are equivalent to $G$ being topologically principal if $G$
is second-countable. We present our structure theorems for the Steinberg algebra $A(G)$
in Section~\ref{sec:algebra}. In Section~\ref{sec:Cstar} we prove $C^*$-algebraic
versions of these results. We choose to pay the price of more-technical statements in
order to describe how our techniques apply to non-amenable groupoids. In a short examples
section we indicate why our techniques cannot be adapted to characterise simplicity of
the reduced $C^*$-algebra of an \'etale groupoid and why our results do not extend
readily to twisted groupoid $C^*$-algebras. We also provide an example of a non-\'etale
groupoid in which every unit has infinite isotropy but no open set consists entirely of
isotropy. By changing the topology, we also construct an \'etale groupoid with totally
disconnected unit space (which is not second-countable) with the same property. We finish
by relating our results to those of Exel-Vershik \cite{EV06} and of Exel-Renault
\cite{ER07}.

\smallskip

\textbf{Acknowledgements.} Thanks to Iain Raeburn, Astrid an Huef, and Dana Williams for
a number of helpful conversations. Further thanks to Dana for his helpful and
constructive comments on a preprint of the paper. Thanks also to Alex Kumjian and Paul
Muhly for very helpful email correspondence.

\section{Preliminaries}

If $X$ is a topological space and $D \subseteq X$, then we shall write $D^\circ$ for the interior
$\bigcup\{U \subseteq D : U\text{ is open in }X\}$ of $D$.

A groupoid $G$ is a small category in which every morphism has an inverse.  When $G$ is
endowed with a topology under which the range, source, and composition maps are
continuous, $G$ is called a topological groupoid. We say $G$ is \emph{\'etale} if $r$ and
$s$ are local homeomorphisms. It then follows that $\go := \{\gamma\gamma^{-1} : \gamma
\in G\}$ is open in $G$.  If $G$ is Hausdorff, then $\go$ is also closed in $G$. For a
more detailed description of \'etale groupoids, see
\cite{Paterson:Groupoidsinversesemigroups99}.

A subset $B$ of $G$ such that $r$ and $s$ both restrict to homeomorphisms of $B$ is called a
\emph{bisection} of $G$. If $G$ is a locally compact, Hausdorff, \'etale groupoid, then there is a
base for the topology on $G$ consisting of open bisections with compact closure (we call such sets precompact in this paper). As demonstrated in \cite{CFST, Steinberg2010}, if $\go$ is totally disconnected and $G$ is locally compact, Hausdorff, and \'etale, then there is base for the topology on $G$ consisting of compact open bisections.

For  subsets $D, E$ of $\go$, define
\[
G_D:=\{\gamma\in G: s(\gamma)\in D\},\quad
G^E:=\{\gamma\in G: r(\gamma)\in E\}\quad\text{and}\quad
G_D^E := G^E\cap G_D.
\]
In a slight abuse of notation, for $u,v\in \go$ we denote $G_u := G_{\{u\}}$, $G^v := G^{\{v\}}$
and $G^v_u := G^v\cap G_u$. The \emph{isotropy group} at a unit $u$ of $G$ is the group
$G^u_u=\{\gamma \in G : r(\gamma) =s(\gamma) = u\}$. We say $u$ has trivial isotropy if
$G^u_u=\{u\}$. The \emph{isotropy subgroupoid} of a groupoid $G$ is $\Iso(G) := \bigcup_{u\in \go}
G^u_u$. Since $r$ and $s$ are continuous, the isotropy subgroupoid of $G$ is a closed subset of
$G$.

A subset $D$ of $\go$ is called \emph{invariant} if $s(\gamma) \in D \implies r(\gamma) \in D$ for
all $\gamma \in G$. Since $G$ contains inverses, this is equivalent to saying that $D= \{r(\gamma)
: s(\gamma) \in D\} = \{s(\gamma) : r(\gamma) \in D\}$; hence $G_D=G^D$, and $G_D$ is a groupoid
with unit space $D$. Also, $D$ is invariant if and only if its complement is invariant.

For subsets $S$ and $T$ of $G$, define $ST = \{\gamma\alpha : \gamma \in S, \alpha \in T,
\text{ and } s(\gamma) = r(\alpha)\}$.

\begin{dfn}
Let $G$ be a locally compact, Hausdorff groupoid. We  say that $G$ is \emph{topologically
principal} if $\big\{u \in \go : G^u_u = \{u\}\big\}$ is dense in $\go$. We  say that $G$
is \emph{minimal} if $\go$ has no nontrivial open invariant subsets. We say $G$ is
\emph{effective} if the interior of $\Iso(G) \setminus \go$ is empty.
\end{dfn}

\begin{rmk}\label{rmk:top prin}
To relate our later results to those of Thomsen \cite{Thomsen:AIF10}, we observe that a
Hausdorff, \'etale groupoid $G$ is topologically principal if and only if each open invariant subset of
$\go$ contains a point with trivial isotropy. To see this, note that the ``only if" implication is
trivial. So suppose that every open invariant set contains a point with trivial isotropy, and fix
an open subset $U$ of $\go$. Then $r(G_U)$ is an open invariant set, so contains a point $u$ with
trivial isotropy. Fix $\gamma \in G_U$ with $r(\gamma) = u$.  Since $G_{s(\gamma)}^{s(\gamma)} =
\gamma^{-1} G^u_u \gamma = \gamma^{-1} \{u\} \gamma = r(\gamma)$, we see that $s(\gamma)$ has
trivial isotropy.  That is, the set $U$ contains a point with trivial isotropy. So $G$ is
topologically principal.

It follows immediately from this that if a minimal groupoid $G$ has a unit with trivial isotropy
then it is topologically principal.
\end{rmk}

\begin{rmk}\label{rmk:terminology}
In groupoid literature, the condition which we are calling \emph{topologically principal}
has gone under this name and a number of others, including ``essentially free,''
``topologically free,'' and ``essentially principal."  We have chosen the one we believe
to be least open to misinterpretation: The usage of the term ``principal" for groupoids
with everywhere-trivial isotropy seems uncontroversial, so ``topologically principal" is
suggestive. Our choice also seems to match what Renault himself has settled on
\cite{Renault:IMSB08, Renault:xx09}.

Similarly, our usage of the terms \emph{minimal} and \emph{effective} seem to be standard
(see, for example, \cite[Definition~I.4.1]{Ren80} and
\cite[Definition~3.4]{Renault:IMSB08}) but are possibly not universal.
\end{rmk}

\section{Topologically Principal Groupoids}\label{sec:groupoid}

The following lemma establishes the equivalent conditions that we use in
Theorem~\ref{thm:alg_simple}  to characterise simplicity of $A(G)$.

\begin{lemma}\label{lem:equiv conditions}
Let $G$ be a locally compact, Hausdorff, \'etale groupoid. The following are equivalent:
\begin{enumerate}
\item\label{it:nonunit isotropy} $G$ is effective;
\item\label{it:isotropy formulation} the interior of $\Iso(G)$ is $\go$;
\item\label{it:3.1} for every nonempty open bisection $B \subseteq G\setminus \go$,
    there exists $\gamma \in B$ such that $s(\gamma) \not= r(\gamma)$;
\item\label{it:ClaireLem} for every compact $K\subseteq G\setminus \go$ and every
    nonempty open $U \subseteq \go$, there exists a nonempty open subset $V \subseteq
    U$ such that $V KV = \emptyset$.
\end{enumerate}
If $G$ is topologically principal, then $G$ is effective. If $G$ is second-countable and
effective, then $G$ is topologically principal.
\end{lemma}

\begin{proof}[Proof of Lemma~\ref{lem:equiv conditions}]
Since $G$ is Hausdorff and \'etale , $\go$ is both open and closed in $G$. So the
interior $S^\circ$ of any subset $S$ of $G$ is equal to the disjoint union $(S \cap
\go)^\circ \cup (S \setminus \go)^\circ$. Thus~(\ref{it:isotropy formulation}) is
equivalent to~(\ref{it:nonunit isotropy}).

We have \mbox{$(\ref{it:nonunit isotropy})\implies(\ref{it:3.1})$} because open bisections are in
particular open sets. That $G$ is \'etale also implies that the collection of all open bisections
of $G$ form a base for the topology on $G$. In particular, every open set
contains an open bisection, giving \mbox{$(\ref{it:3.1})\implies(\ref{it:nonunit
isotropy})$}.

To see (\ref{it:ClaireLem})~implies~(\ref{it:3.1}), we prove the contrapositive. Suppose
that~(\ref{it:3.1}) does not hold, and fix an open bisection $B_0\subseteq G\setminus G\unit$ such
that $r(\gamma)=s(\gamma)$ for all $\gamma\in B_0$. That is, $B_0 \subseteq \Iso(G)$. By shrinking
if necessary, we may assume that $B_0$ is precompact.  Since $G$ is locally compact and Hausdorff, it is a regular topological space (that is, points can be separated from compact sets by disjoint open sets).  Thus, there is an open subset $B$ of $B_0$ whose closure $K$ is compact and
contained in $B_0$. Let $U=r(B)$, and fix a nonempty open subset $V$ of $U$.   Since $K \subseteq
\Iso(G)$, we have $VK = KV$, and in particular $VKV \not= \emptyset$. Hence~(\ref{it:ClaireLem})
does not hold.

To show that (\ref{it:3.1})~implies~(\ref{it:ClaireLem}), we begin with a claim.

\begin{claim} \label{clm r s disjoint}
Suppose that $B \subseteq G \setminus \go$ is an open bisection and that $\gamma \in B \setminus
\Iso(G)$. Then there is an open set $V \subseteq r(B)$ such that $\gamma \in VB$ and $s(VB) \cap V
= \emptyset$.
\end{claim}
\begin{proof}[Proof of Claim~\ref{clm r s disjoint}]
Since $r(\gamma)\neq s(\gamma)$ and $G$ is Hausdorff, there exist open neighbourhoods $W$ of
$r(\gamma)$ and $W'$ of $s(\gamma)$ such that $W\cap W'=\emptyset$.  Let $V := W \cap r(BW')$.
Notice that $r(\gamma) \in V$ so $V$ is not empty.  Then $\gamma \in VB$, and since $B$ is a
bisection, $s(VB) = s(WB \cap BW') \subseteq W'$ and hence is disjoint from $V \subseteq W$.
\hfil\penalty100\hbox{}\nobreak\hfill\hbox{\qed\ Claim~\ref{clm r s disjoint}}
\renewcommand{\qed}{}\end{proof}

Now suppose~(\ref{it:3.1}), and fix a compact $K\subseteq G\setminus \go$ and an open $U \subseteq
\go$. We construct a nonempty open set $V\subseteq U$ such that $VKV=\emptyset$. If $U$ is not a
subset of $r(K)$, then $V = U \setminus r(K)$ will suffice, so suppose that $U \subseteq r(K)$.
Because $G$ is regular and $\go$ is open, there is a base for the topology on $G \setminus \go$
consisting of precompact open bisections whose closures are themselves contained in open bisections
which do not intersect $\go$. Since $K$ is compact, we may cover $K$ by a finite set $\mathcal{B}$
of such precompact open bisections. For each $B\in\Bb$, fix an open bisection $C_B$ such that
$\overline{B}\subseteq C_B\subseteq G\setminus \go$. For each $B \in \Bb$ the set $UC_B$ is an open
bisection which does not intersect $\go$, and the $r(B)$ cover $U$ so at least one $U C_B$ is
nonempty. So~(\ref{it:3.1}) implies that there exists $\gamma \in \bigcup_{B \in \Bb} UC_B
\setminus \Iso(G)$. Let $\Ff := \{B \in \Bb : \gamma \in UC_B\}$. For each $B \in \Ff$,
Claim~\ref{clm r s disjoint} yields an open set $V_B \subseteq r(UC_B)$ such that $s(V_B C_B) \cap
V_B = \emptyset$. Let
\[
V := U \cap \Big(\bigcap\{V_B : B \in \Ff\}\Big)
    \setminus \Big(\bigcup\{\overline{r(B')} : B'\in \Bb \setminus \Ff\}\Big).
\]
Then $V$ is open by definition, and nonempty because it contains $r(\gamma)$.

Fix $\alpha \in VK$; we must show $s(\alpha)\notin V$. Since $\alpha\in K$ and $\Bb$ is
a cover of $K$, we have $\alpha\in B$ for some $B\in \Bb$.  Also, since $r(\alpha)\in
V$, we have $B\in \Ff$.  Hence
\[
s(\alpha) \in s(VB) \subseteq s(V_B B),
\]
and $s(V B) \cap V\subseteq s(V_B B) \cap V_B=\emptyset$. Therefore $s(\alpha) \not\in V$. Thus
(\ref{it:3.1})~implies~(\ref{it:ClaireLem}).

The final two statements follow from \cite[Proposition~3.6]{Renault:IMSB08} since every
locally compact Hausdorff space has the Baire property.
\end{proof}

\begin{rmk}
The final assertion of Lemma~\ref{lem:equiv conditions} need not hold if $G$ is not
second-countable (see Example~\ref{ex:etale ex}) or if $G$ is not \'etale (see
Example~\ref{ex:2ndctble ex}).
\end{rmk}

\section{Simplicity of Steinberg algebras}\label{sec:algebra}

In this section, we consider locally compact, Hausdorff, \'{e}tale groupoids with totally
disconnected unit spaces. This puts us in the setting of \cite{CFST}. For such a groupoid
$G$, let
\[
    A(G) := \lsp\{1_B : B \text{ is a compact open bisection}\}
\]
as in \cite{CFST}.  For $f,g \in A(G) \subseteq C_c(\Gg)$, define
\begin{gather}
    f^*(\gamma) = \overline{f(\gamma^{-1})} \label{eq:multiplication}; \text{ and}\\
    (f*g)(\gamma) = \sum_{r(\alpha) = r(\gamma)} f(\alpha) g(\alpha^{-1}\gamma)
        = \sum_{\alpha\beta = \gamma} f(\alpha) g(\beta).\label{eq:involution}
\end{gather}
Under these operations and pointwise addition and scalar multiplication, $A(G)$ is a
$^*$-subalgebra of $C_c(\Gg)$. It coincides with the complex inverse semigroup algebra
$\CC G$ introduced in \cite{Steinberg2010}.\footnote{We prefer the notation $A(G)$
because Steinberg's notation $\CC G$ suggests the free $\CC$-module with basis $G$, which
is substantially larger. To avoid clashing with Steinberg's notation, we use $\FF(W)$ for
the free complex module with basis $W$.} We call $A(G)$ the Steinberg algebra of $G$.

\begin{theorem} \label{thm:alg_simple}
Let $G$ be a locally compact, Hausdorff, \'etale groupoid such that $\go$ is totally
disconnected. Then $A(G)$ is simple if and only if $G$ is both effective and minimal.
\end{theorem}

Our proof was guided by that of Theorem~5.14 in \cite{ArandaClarkEtAl:xx11}. However,
their arguments rely heavily on the underlying higher-rank graph structure so our approach looks
very different.  The first step is to prove that the Cuntz-Krieger uniqueness theorem for
$A(G)$ \cite[Theorem~5.2]{CFST} still holds if we replace the hypothesis that $G$ is
topologically principal with the hypothesis that $G$ is effective.

\begin{lemma}
\label{lem:CKuniqueness} Let $G$ be a locally compact, Hausdorff, \'etale groupoid with
totally disconnected unit space. Suppose that $G$ is effective and that $I$ is a
nontrivial ideal of $\ag$. Then there is a compact open subset $V \subseteq \go$ such
that $1_V \in I$.
\end{lemma}
\begin{proof}
Fix $b \in I \setminus \{0\}$. Let $c := b^* * b$. For $u \in \go$ we have
\[
c(u) =\sum_{\gamma \in G_u} b^*(\gamma^{-1}) b(\gamma) = \sum_{\gamma \in G_u}
 \overline{b(\gamma)} b(\gamma) \ge \max_{\gamma \in G_u} |b(\gamma)|^2.
\]
In particular, the function
\[c_0 := \begin{cases}
c(\gamma)&\text{if } \gamma \in \go;\\0 & \text{otherwise}
        \end{cases}\]
is nonzero.  Because $\go$ is both open and closed, $c_0 \in A(G)$.

Using Lemma~3.6 of \cite{CFST} we may write
\[
    c_0 = \sum_{U \in \Uu} a_U 1_U,
\]
where $\Uu$ is a collection of mutually disjoint, nonempty compact open subsets of $\go$, and each
$a_U$ is nonzero. Let $K$ be the support of $c-c_0$.  Notice that $K \subseteq G \setminus \go$.

Fix $U \in \Uu$. Since $\Iso(G)^\circ = \go$, the implication \mbox{$(\ref{it:isotropy
formulation})\implies(\ref{it:ClaireLem})$} of Lemma~\ref{lem:equiv conditions} implies
that there exists a nonempty open set $V\subseteq U$ such that $VKV=\emptyset$. Since $G$
has a basis of compact open sets, we can assume $V$ is also compact.

For $\gamma\in G$ we have
\[
    (1_V (c - c_0) 1_V)(\gamma)=1_V(r(\gamma)) (c - c_0)(\gamma) 1_V(s(\gamma))=0.
\]
So $1_V c 1_V = 1_V c_0 1_V = a_{U} 1_V$. Hence $1_V \in I$.
\end{proof}

Another key ingredient in our proof of Theorem~\ref{thm:alg_simple} is the following
generalisation of the \emph{infinite-path representation} of a Kumjian-Pask algebra as defined on
page~9 of \cite{ArandaClarkEtAl:xx11}. In our setting, the infinite-path space becomes the unit
space of $G$. In fact, the construction of \cite{ArandaClarkEtAl:xx11}
works for any invariant subset $W$ of $\go$. Given such a set $W$, we write $\FF(W)$
for the free (complex) module with basis $W$. We use these representations to construct nontrivial
ideals of $A(G)$  when there exists either a nontrivial open invariant subset
of $\go$ or a nonempty open subset of $\Iso(G)\setminus\go$.

\begin{prop}\label{prp:module rep}
Let $G$ be a locally compact, Hausdorff, \'etale groupoid with totally disconnected unit space and
let $W$ be an invariant subset of $\go$.
\begin{enumerate}
\item \label{it:module rep 1} For every compact open bisection $B \subseteq G$, there is a
    unique function $f_B : \go \to \FF(W)$ that has support contained in $s(B)$ and satisfies
    $f_B(s(\gamma)) = r(\gamma)$ for all $\gamma \in B$.
\item \label{it:module re 2} There is a unique representation $\pi_W : A(G) \to
    \operatorname{End}(\FF(W))$ such that $\pi_W(1_B)u = f_B(u)$ for every compact open
    bisection $B$ and all $u \in W$.
\end{enumerate}
\end{prop}
\begin{proof}
Let $B$ be a compact, open bisection in $G$.  The formula $s(\gamma) \mapsto r(\gamma)$
for $\gamma$ in $B$ specifies a well-defined homeomorphism from $s(B)$ to $r(B)$.  Thus,
the function $f_B$ can be defined as stated in~(\ref{it:module rep 1}). To
prove~(\ref{it:module re 2}), first notice that the universal property of the free module
$\FF(W)$ implies that there is an element $t_B \in \operatorname{End}(\FF(W))$ extending
$f_B|_{W}$. Let $c : G \to \{e\}$ be the trivial cocycle. Then every bisection of $G$ is
$e$-graded under $c$, so the set $\BcoG$ of \cite[Definition~3.10]{CFST} is the set of
all compact open bisections of $G$. We claim that the collection $\{t_B : B \in \BcoG\}$
gives a representation of $\BcoG$ in $\operatorname{End}(\FF(W))$ as defined in
Definition~3.10 of \cite{CFST}.

To prove our claim, we must verify that:
\begin{enumerate}\renewcommand{\theenumi}{R\arabic{enumi}}
\item\label{it:zero} $t_\emptyset = 0$;
\item\label{it:multiplicative} $t_Bt_D = t_{BD}$ for all compact open bisections $B$ and $D$;
    and
\item\label{it:additive} $t_B + t_D = t_{B \cup D}$ whenever $B$ and $D$ are disjoint compact
    open bisections such that $B \cup D$ is a bisection.
\end{enumerate}
It is straightforward to check that each of these conditions holds for the functions $f_B$, and
hence for the endomorphisms $t_B$ as well.

Now, the universal property of $A(G)$, stated in Theorem 3.11 of \cite{CFST}, gives a
unique homomorphism $\pi_W:\ag \to \operatorname{End}(\FF(W))$ such that $\pi_W(1_B) =
t_B$ for all $B \in \BcoG$. The homomorphism $\pi_W$ is nonzero because $t_B$ is nonzero
whenever $s(B) \cap W \neq \emptyset$. It satisfies $\pi_W(1_B)u = f_B(u)$ because each
$t_B$ extends $f_B|_W$.
\end{proof}

\begin{prop}
\label{prp:pi_inj_implies_3.1} Let $G$ be a locally compact, Hausdorff, \'etale groupoid
with totally disconnected unit space, and let $\pi:= \pi_{\go} : A(G) \to \End(\FF(\go))$
be the homomorphism of Proposition~\ref{prp:module rep}. Then $\pi$ is injective if and
only if $G$ is effective.
\end{prop}
\begin{proof}
First suppose that $G$ is effective. Since $\pi(1_V) \not= 0$ for all compact open $V
\subseteq \go$, $\pi$ is injective by the contrapositive of Lemma~\ref{lem:CKuniqueness}
applied to $I = \ker(\pi)$.

Now suppose that $G$ is not effective. By \mbox{$(\ref{it:nonunit
isotropy})\iff(\ref{it:3.1})$} of Lemma~\ref{lem:equiv conditions}, there exists a
nonempty compact open bisection $B \subseteq G \setminus \go$ so that for every $\gamma
\in B$, $r(\gamma) = s(\gamma)$. Hence $B \neq s(B)$ but $f_B= f_{s(B)}$, where $f_B$ is
defined in Proposition~\ref{prp:module rep}(\ref{it:module rep 1}).  Thus
$\pi_{\go}(1_B)= \pi_{\go}(1_{s(B)})$ giving $1_B - 1_{s(B)} \in \ker(\pi_{\go})$. Since
$B \neq s(B)$ we have $1_B - 1_{s(B)} \not= 0$, so $\ker(\pi_{\go}) \neq \{0\}$.
\end{proof}

\begin{prop}
\label{prop:alg_nois_iff_ideal}  Let $G$ be a locally compact, Hausdorff, \'etale
groupoid with totally disconnected unit space. Then $G$ is minimal if and only if every
nonzero $f \in A(G)$ such that $\supp f \subseteq \go$ generates $A(G)$ as an ideal.
\end{prop}
\begin{proof}
Suppose $G$ is minimal. Fix $f \in A(G) \setminus \{0\}$ such that $\supp f \subseteq
\go$. Let $I$ be the ideal of $A(G)$ generated by $f$. Fix $g \in A(G)$; we must show
that $g \in I$. Since $f$ is nonzero and locally constant \cite[Lemma~3.4]{CFST}, there
exist $c \in \CC \setminus \{0\}$ and a compact open $U \subseteq \go$ so that $f|_U
\equiv c$. Then $1_U = \frac{1}{c} 1_U * f \in I$. Let $K := r(\supp(g)) \subseteq \go$.
Then $K$ is compact and open by \cite[Lemma~3.2]{CFST}. Since $s(G^U)$ is a nonempty open
invariant set, it is all of $\go$. Therefore $K \subseteq s(G^U)$. So for each $u \in K$,
there exists $\gamma_u$ with $r(\gamma_u) \in U$ and $s(\gamma_u) = u$. For each $u$, let
$B_u$ be a compact open bisection containing $\gamma_u$ such that $r(B_u) \subseteq U$
and $s(B_u) \subseteq K$. Then $1_{s(B_u)} = 1^*_{B_u} * 1_{U} * 1_{B_u}$ belongs to $I$.
Since $K$ is compact, there is a finite subset $\{v_1, \dots, v_n\}$ of $K$ such that
$\{s(B_{v_i}) : 1 \le i \le n\}$ covers $K$. By disjointification of the collection
$\{s(B_{v_i}) : 1 \le i \le n\}$ (see \cite[Remark~2.5]{CFST}), we may assume that the
$s(B_{v_i})$ are mutually disjoint. For each $i$, the function $k_i := 1_{s(B_{v_i})} \le
1_{s(B_u)}$ belongs to $I$, so $1_K = \sum^n_{i=1} k_i \in I$. Hence $g = 1_K * g \in I$.

Conversely, suppose $G$ is not minimal. Let $U$ be a nontrivial open invariant subset of $\go$.
Then the complement $W := \go\setminus U$ is itself an invariant subset of $\go$.  Let $\pi_{W} :
A(G) \to \operatorname{End}\FF(W)$ be the nonzero homomorphism of Proposition~\ref{prp:module rep}.
The kernel of $\pi_W$ is a proper ideal of $A(G)$. To complete the proof, it suffices to show that
$\ker(\pi_W) \not= \{0\}$. To see this, let $B \subseteq U$ be a compact open set. Then $1_B \in
\ker \pi_W \setminus \{0\}$.
\end{proof}

\begin{proof}[Proof of Theorem~\ref{thm:alg_simple}]
Suppose $A(G)$ is simple.  Then $\pi_{\go}$ is injective so
Proposition~\ref{prp:pi_inj_implies_3.1} implies that $G$ is effective. Since $A(G)$ is
simple, every function with support contained in $\go$ generates $A(G)$ as an ideal.
Hence, $G$ is minimal by Proposition~\ref{prop:alg_nois_iff_ideal}.

Conversely, suppose that $G$ is effective and minimal. Fix a nonzero ideal $I$ in $A(G)$.
Lemma~\ref{lem:CKuniqueness} implies that there is a compact open subset $V \subseteq
\go$ such that $1_V \in I$. Proposition~\ref{prop:alg_nois_iff_ideal} implies that the
ideal generated by $1_V$ is all of $A(G)$, so $I = A(G)$.
\end{proof}

\section{Simplicity of groupoid \texorpdfstring{$C^*$}{C*}-algebras}\label{sec:Cstar}

For details of the following, see, for example, \cite{Ren80} or
\cite{Paterson:Groupoidsinversesemigroups99}. Let $G$ be a second-countable, locally compact,
Hausdorff, \'etale groupoid. The formulas \eqref{eq:multiplication}~and~\eqref{eq:involution} for
convolution and involution on $\ag$ described in the preceding section also define a convolution
and involution on $C_c(G)$.  With these operations, and pointwise addition and scalar
multiplication, $C_c(G)$ is a complex $^*$-algebra. The $I$-norm on $C_c(G)$ defined by
\[
\|f\|_I = \sup_{u \in \go} \max\Big\{\sum_{\gamma \in G_u} |f(\gamma)|, \sum_{\gamma \in G^u} |f(\gamma)|\Big\}
\]
is a $^*$-algebra norm (see Proposition~II.1.4 of \cite{Ren80}) but not typically a
$C^*$-norm. The full norm on $C_c(G)$ is defined by
\[
    \|f\| := \sup\{\|\pi(f)\| : \text{$\pi$ is an $I$-norm-bounded $^*$-representation of $C_c(G)$}\},
\]
and $C^*(G)$ is defined to be the completion of $C_c(G)$ in the full norm.

There is a distinguished family of $I$-norm-bounded representations of $C_c(G)$, called
the regular representations; each is indexed by a $u\in\go$  and denoted $\Ind_u$.
Specifically, the regular representation $\Ind_u$  is the representation of $C_c(G)$ on
$\ell^2(G_u)$ implemented by convolution. That is, $\Ind_u(f)\delta_\gamma = \sum_{\beta
\in G^{r(\gamma)}} f(\beta^{-1}\gamma)\delta_\beta$. The reduced $C^*$-algebra $C^*_r(G)$
is the completion of $C_c(G)$ in the reduced norm $\|f\|_r = \sup_{u \in \go}
\|\Ind_u(f)\|$.  The reduced norm is dominated by the full norm, so $C^*_r(G)$ is a
quotient of $C^*(G)$.

We can now state our main theorem.

\begin{theorem}
\label{thm:main} Let $G$ be a second-countable, locally compact, Hausdorff, \'etale
groupoid. Then $C^*(G)$ is simple if and only if all of the following conditions are
satisfied:
\begin{enumerate}
\item\label{it:metrically amenable} $C^*(G) = C^*_r(G)$;
\item\label{it:discretely trivial} $G$ is topologically principal; and
\item\label{it:minimal} $G$ is minimal.
\end{enumerate}
\end{theorem}

Our proof of Theorem~\ref{thm:main} relies on the following adaptation of the augmentation
representation of a discrete group. Let $G$ be a groupoid as in Theorem~\ref{thm:main}. For each $u
\in \go$, let $[u]$ denote the orbit of $u$ under $G$; that is $[u] = r(G_u)$.

\begin{prop}
\label{prp:pi_u} Let $G$ be a second-countable, locally compact, Hausdorff, \'etale groupoid.
Fix $u \in \go$. There is a unique representation
$\pi_{[u]}$ of $C^*(G)$ on $\ell^2([u]) = \clsp\{\delta_v : v \in [u]\}$ such that for each $f \in
C_c(G)$ and $v \in [u]$,
\begin{equation}\label{eq:pi_u}
    \pi_{[u]}(f)\delta_v := \sum_{\gamma \in G_v} f(\gamma) \delta_{r(\gamma)}.
\end{equation}
\end{prop}

\begin{rmk}
In equation~\ref{eq:pi_u}, we described $\pi_{[u]}$ in terms of the canonical orthonormal
basis for $\ell^2([u])$. For an alternative description, let $\mu$ be the measure $\mu(V)
:= |V \cap [u]|$ on $\go$. Then $\pi_{[u]}$ is the representation on $L^2(\go, \mu)$
obtained from the usual left-action of $C_c(G)$ on $C_c(\go)$ --- namely $f \cdot \phi(u)
= \sum_{\gamma \in G^u} f(\gamma) \phi(s(\gamma))$.
\end{rmk}

\begin{proof}[Proof of Proposition~\ref{prp:pi_u}]
For $f \in C_c(G)$ and a finite linear combination $h=\sum_{v\in [u]} h_v
\delta_v$, let $f \cdot h$ be the vector $\sum_{v \in [u]}
\sum_{\gamma \in G_v} f(\gamma) h_v \delta_{r(\gamma)}$. Then $h \mapsto f\cdot h$ is linear, and
$f \cdot \delta_v$ is equal to the right-hand side of~\eqref{eq:pi_u}. The following is adapted
directly from the proof of \cite[Proposition~II.1.7]{Ren80}. Fix $f \in C_c(G)$.
For $v,w \in [u]$, we have
\begin{equation}\label{eq:adjointable}
(f \cdot \delta_v | \delta_w)
    = \sum_{\gamma \in G_v} (f(\gamma) \delta_{r(\gamma)} | \delta_w)
    = \sum_{\gamma \in G_v^w} f(\gamma)
    = \sum_{\gamma \in G_w} (\delta_v | \overline{f(\gamma^{-1})} \delta_w)
    = (\delta_v | f^* \cdot \delta_w).
\end{equation} Since $k \mapsto f \cdot k$ is linear on $\lsp\{\delta_v : v \in [u]\}$,
it follows that $(f \cdot k | k') = (k | f^* \cdot k')$ for all $k,k' \in C_c([u])$. In particular,
for a finite linear combination $h=\sum_{v\in [u]} h_v \delta_v$,
\begin{align*}
\|f \cdot h\|^2
    &= ((f^*f) \cdot h | h) \\
    &= \Big|\sum_{\gamma \in G_{[u]}} \overline{(f^*f)(\gamma) h_{s(\gamma)}} h_{r(\gamma)}\Big| \\
    &\le \sum_{\gamma \in G_{[u]}} |(f^*f)(\gamma)| |h_{s(\gamma)}| |h_{r(\gamma)}| \\
    &= \sum_{\gamma \in G_{[u]}} (|(f^*f)(\gamma)|^{1/2} |h_{s(\gamma)}|)
(|(f^*f)(\gamma)|^{1/2} |h_{r(\gamma)}|). \\
\intertext{So the Cauchy-Schwarz inequality gives}
\|f \cdot h\|^2
    &\le \Big(\sum_{\gamma \in G_{[u]}} |(f^*f)(\gamma)| |h_{s(\gamma)}|^2\Big)^{1/2}
        \Big(\sum_{\beta \in G_{[u]}} |(f^*f)(\beta)| |h_{r(\beta)}|^2\Big)^{1/2} \\
    &= \Big(\sum_{v \in [u]} \Big(\sum_{\gamma \in G_v} |(f^*f)(\gamma)|\Big) |h_v|^2\Big)^{1/2}
        \Big(\sum_{w \in [u]} \Big(\sum_{\beta \in G^w} |(f^*f)(\beta)|\Big) |h_w|^2\Big)^{1/2} \\
    &\le \|(f^*f)\|_I \|h\|^2.
\end{align*}
Proposition~II.1.4 of \cite{Ren80} (or direct calculation) shows that $\|f^*f\|_I \le \|f\|^2_I$, and
it follows that $\|f\cdot h\| \le \|f\|_I\|h\|$.

Thus, for each $f \in C_c(G)$ the formula~\eqref{eq:pi_u} determines a bounded linear operator
$\pi_{[u]}(f)$ on $\ell^2([u])$, and the map $f \mapsto \pi_{[u]}(f)$ is bounded with respect to
the $I$-norm. By definition of the norm on $C^*(G)$, it therefore remains only to show that
$\pi_{[u]}$ is a $^*$-homomorphism from $C_c(G)$ to $\Bb(\ell^2([u]))$. The
calculation~\eqref{eq:adjointable} shows that $\pi_{[u]}(f)^* = \pi_{[u]}(f^*)$. For $f,g \in
C_c(G)$ and $v \in [u]$,
\begin{align*}
\pi_{[u]}(f*g)\delta_v
    &= \sum_{\gamma \in G_v} (f*g)(\gamma)\delta_{r(\gamma)}
    = \sum_{\alpha\beta \in G_v} f(\alpha)g(\beta) \delta_{r(\alpha)} \\
    &= \sum_{\beta \in G_v} \sum_{\alpha \in G_{r(\beta)}} f(\alpha)g(\beta)\delta_{r(\alpha)}
    = \sum_{\beta \in G_v} \pi_{[u]}(f)(g(\beta)\delta_{r(\beta)})
    = \pi_{[u]}(f) \pi_{[u]}(g)\delta_v.
\end{align*}
Hence $\pi_{[u]}$ is a $^*$-homomorphism as required.
\end{proof}

\begin{rmk}
\label{rmk:augmentation} The direct sum  $\epsilon_G := \bigoplus_{[u] \in \go/G} \pi_{[u]}$ of $G$
is faithful on $C_0(\go)$.  To see this, fix $f \in C_c(\go) \setminus \{0\}$ and
$u \in \go$ such that $f(u) \not= 0$. Then $\|\epsilon_G(f)\| \ge \|\pi_{[u]}(f)\delta_u\| =
\|f(u)\delta_u\| \not= 0$. If $G$ is a (discrete) group, then $\epsilon_G$ is just the
$1$-dimensional representation of $C^*(G)$ induced by the unitary representation $\epsilon : g
\mapsto 1$ of $G$, sometimes called the \emph{augmentation representation} of $G$.
\end{rmk}


\begin{prop}
\label{it:Iso hypothesis} Let $G$ be a second-countable, locally compact, Hausdorff, \'etale groupoid.
\begin{enumerate}
\item\label{it:iso->CKUT} Suppose that $G$ is topologically principal. Then every
    ideal $I$ of the reduced $C^*$-algebra $C_r^*(G)$ satisfies $I \cap C_c(\go)
    \not= \{0\}$.
\item\label{it:CKUT->iso} Suppose that every ideal of the full $C^*$-algebra $C^*(G)$
    satisfies $I \cap C_0(\go) \not= \{0\}$. Then $G$ is topologically principal.
\end{enumerate}
\end{prop}
\begin{proof}
(\ref{it:iso->CKUT}) Since $G$ is topologically principal, Lemma~\ref{lem:equiv
conditions} implies that it is effective. The result then follows from
\cite[Theorem~4.4]{Exel:PAMS10}\footnote{Exel uses the term \emph{essentially principal}
for what we call \emph{effective} (see \cite[p.~897]{Exel:PAMS10})} (see also
\cite[Corollary~4.9]{Renault:JOT91}).

(\ref{it:CKUT->iso}) We prove the contrapositive. Suppose that $G$ is not topologically
principal. Then Lemma~\ref{lem:equiv conditions} implies that there is an open bisection
$B$ in $G\setminus \go$ consisting entirely of isotropy. Let $\epsilon_G$ be the direct
sum representation defined in Remark~\ref{rmk:augmentation}.  We show that
$\ker(\epsilon_G)$ is a nontrivial ideal in $C^*(G)$ that does not intersect $C_0(\go)$.
By Remark~\ref{rmk:augmentation}, $\ker(\epsilon_G) \cap C_0(\go) = \{0\}$ so it suffices
to construct a nonzero element of $\ker{\epsilon_G}$.

For each $u \in s(B)$, let $\gamma_u$ be the unique element in $B$ such that $s(\gamma_u) = u$. Fix
a nonzero function $f \in C_c(G)$ such that $\supp(f) \subseteq B$, and define $f_0 \in C_c(\go)$
by
\[
f_0(u):=\begin{cases} f(\gamma_u) &\text{if~} u\in s(B),\\
				0 &\text{otherwise.}\end{cases}
\]
Since $B \cap \go = \emptyset$ and since $f \not= 0$, we have $f - f_0 \not= 0$.  We claim that
$\epsilon_G(f - f_0) = 0$; that is,  $\pi_{[u]}(f - f_0) = 0$ for all $u \in \go$. To see
this, fix $u \in \go$ and $v \in [u]$. Then
\[
\pi_{[u]}(f - f_0)\delta_v
    = \sum_{\gamma \in G_v} f(\gamma)\delta_{r(\gamma)} - \sum_{\alpha \in G_v} f_0(\alpha)\delta_{r(\alpha)}.
\]
If $v \not\in s(B)$, then $f(\gamma) = f_0(\alpha) = 0$ for all $\gamma,\alpha \in G_v$, so
$\pi_{[u]}(f - f_0)\delta_v = 0$. Suppose that $v \in s(B)$. Since $f_0$ is supported on units and
$f$ is supported on $B$,
\[
\sum_{\gamma \in G_v} f(\gamma)\delta_{r(\gamma)} - \sum_{\alpha \in G_v} f_0(\alpha)\delta_{r(\alpha)}
    = f(\gamma_v)\delta_{r(\gamma_v)} - f_0(v) \delta_v = f(\gamma_v) \delta_{r(\gamma_v)} - f(\gamma_v) \delta_v.
\]
Since $B \subseteq \Iso(G)$, we have $r(\gamma_v) = s(\gamma_v) = v$, and it follows that
$\pi_{[u]}(f - f_0)\delta_v = 0$.
\end{proof}

The following standard lemma is used in the proofs of Proposition~\ref{it:irr hypothesis}
and Corollary~\ref{cor:maincor}.

\begin{lemma}\label{lem:technical}
Let $G$ be a locally compact, Hausdorff, \'etale groupoid. Suppose that $h \in C_c(G)$ is supported
on a bisection $B$ and that $f \in C_c(\go)$. Then $h * f * h^* \in C_c(\go)$ with support
contained in $r(B) \subseteq \go$ and satisfies
\[
(h * f * h^*)(r(\gamma)) = |h(\gamma)|^2 f(s(\gamma))\quad\text{ for all $\gamma \in B$.}
\]
\end{lemma}
\begin{proof}
For $\alpha \in G$, we have
\begin{equation} \label{eq:hfh* formula}
    h*f*h^*(\alpha) = \sum_{\gamma \eta \beta^{-1} = \alpha} h(\gamma) f(\eta) \overline{h(\beta)}.
\end{equation}
Fix $\gamma\eta\beta^{-1} \in G$ with $h(\gamma)f(\eta)\overline{h(\beta)} \not= 0$. Since
$\supp(f)\subseteq \go$, we have $\eta = s(\gamma) = s(\beta)$.  Since $h$ is supported on the
bisection $B$, it follows that $\gamma,\beta \in B$ and $\beta = \gamma$. Hence
$\gamma\eta\beta^{-1} = \gamma s(\gamma) \gamma^{-1} = r(\gamma) \in r(B)$. Thus the sum on the
right of~\eqref{eq:hfh* formula} is zero if $\alpha \not \in r(B)$, and has only one nonzero term
$h(\gamma) f(s(\gamma)) \overline{h(\gamma)} = |h(\gamma)|^2 f(s(\gamma))$ if $\alpha = r(\gamma)
\in r(B)$.
\end{proof}

\begin{prop}
\label{it:irr hypothesis} Let $G$ be a second-countable, locally compact, Hausdorff, \'etale groupoid.
The following are equivalent:
\begin{enumerate}
\item\label{it:irred} $G$ is minimal;
\item\label{it:full ideals} the ideal of $C^*(G)$ generated by any nonzero $f \in C_c(\go)$ is
    $C^*(G)$; and
\item\label{it:red ideals} the ideal of $C_r^*(G)$ generated by any nonzero $f \in C_c(\go)$ is
    $C_r^*(G)$.
\end{enumerate}
\end{prop}
\begin{proof}
\mbox{$(\ref{it:irred})\implies(\ref{it:full ideals})$} and
\mbox{$(\ref{it:irred})\implies(\ref{it:red ideals})$}. Let $f\in C_c(\go)\setminus \{0\}$ and let
$I$ be the ideal of $C^*(G)$ generated by $f$. We claim that $C_c(\go) \subseteq I$. Since $I \cap
C_0(G^{(0)})$ is an ideal of $C_0(G^{(0)})$, it suffices to show that for each $u \in G^{(0)}$,
there exists $g \in I \cap C_0(G^{(0)})$ such that $g(u) \not= 0$. Fix $u \in G^{(0)}$. Let $U :=
\{v \in G^{(0)}: f(v) \not= 0\}$. Then $U$ is nonempty and open, and hence $r(G_U)$ is open because
$s$ is continuous and  the local homeomorphism $r$ is an open map. So $r(G_U)$ is a nonempty open
invariant set, and hence is equal to $G^{(0)}$ because $G$ is minimal. In particular, there exists
$\gamma \in G$ such that $s(\gamma) \in U$ and $r(\gamma) = u$. Fix $h \in C_c(G)$ such that
$\supp(h)$ is contained in a bisection and $h(\gamma) = 1$. Lemma~\ref{lem:technical} implies that
$(h * f * h^*)(u) = |h(\gamma)|^2 f(s(\gamma)) = f(s(\gamma)) \not= 0$. So $g := h
* f * h^*$ belongs to $I \cap C_0(\go)$ with $g(u) = 1$. This proves the claim.

Fix $F \in C_c(G)$. Then any $g \in C_c(\go)$ such that $g|_{r(\supp(F))} \equiv 1$
satisfies $g * F= F$. Hence $C_c(G) \subseteq I$, and so $I = C^*(G)$. Let $q : C^*(G) \to C^*_r(G)$ be the
quotient map. Then the ideal $I_r$ of $C^*_r(G)$ generated by $f$ is $q(I)$. Since $q$ restricts to
the identity map on $C_c(G)$, we have $C_c(G) \subseteq I_r$ as well, and hence $I_r = C^*_r(G)$.

\mbox{$(\ref{it:full ideals})\implies(\ref{it:irred})$} and \mbox{$(\ref{it:red
ideals})\implies(\ref{it:irred})$}. We prove the contrapositive. Suppose that $U$ is a nonempty
proper open invariant subset of $\go$. Fix $f \in C_c(G) \setminus \{0\}$ such that $\supp(f)
\subseteq U$. Then $f \in C_c(\go)$. Fix $u \in \go\setminus U$. Since $\go\setminus U$ is
invariant, $[u] \subseteq \go\setminus U$, so $f(v) = 0$ for all $v \in [u]$. It follows that the
image of $f$ under the regular representation $\Ind_u$ is zero. On the other hand, for any $g \in
C_c(\go)$ such that $g(u) = 1$, we have $\Ind_u(g)\delta_u = g(u)\delta_u \not= 0$. So $\Ind_u$ is
a nonzero representation of $C_c(G)$ with nontrivial kernel. Since $\Ind_u$ extends to each of
$C^*_r(G)$ and $C^*(G)$ it follows that the ideals of each of $C^*(G)$ and $C^*_r(G)$ generated by
$f$ are proper ideals.
\end{proof}

\begin{rmk}
Suppose that $G$ is locally compact, Hausdorff and \'etale. Thomsen observes in
\cite{Thomsen:AIF10} that if $G$ has a unit with trivial isotropy, then $G$ is topologically
principal whenever it is minimal (see Remark~\ref{rmk:top prin}). He then deduces that if $G$ has a
unit with trivial isotropy, then $C^*(G)$ is simple if and only if $G$ is minimal.  We recover this
result from Proposition~\ref{it:Iso hypothesis}(\ref{it:CKUT->iso}) together with
(\ref{it:irred})${}\iff{}$(\ref{it:red ideals}) of Proposition~\ref{it:irr hypothesis}.
\end{rmk}

\begin{proof}[Proof of Theorem~\ref{thm:main}]
Suppose $C^*(G)$ is simple. Then the quotient map from $C^*(G) \to C^*_r(G)$ has trivial kernel and
hence the two coincide. Moreover, $C^*(G)$ is the only nonzero ideal of $C^*(G)$ and $C^*(G)\cap
C_0(\go) \not= \{0\}$ so Proposition~\ref{it:Iso hypothesis} implies that $G$ is topologically
principal. The simplicity of $C^*(G)$ implies that every $f\in C_c(\go)$ generates $C^*(G)$ as an
ideal and so Proposition~\ref{it:irr hypothesis} implies that $G$ is minimal.

Now suppose that $C^*(G) = C^*_r(G)$ and that $G$ is topologically principal and minimal. Fix a
nonzero ideal $I$ in $C^*(G)$. Since $C^*(G) = C^*_r(G)$, Proposition~\ref{it:Iso
hypothesis}(\ref{it:iso->CKUT}) implies there exists a nonzero $f\in C_c(G\unit)\cap I$; and then
\mbox{$(\ref{it:irred})\implies(\ref{it:red ideals})$} of Proposition~\ref{it:irr hypothesis}
implies that the ideal generated by $f$ is $C^*(G)$.  Thus $I=C^*(G)$.
\end{proof}

Corollary~\ref{cor:maincor} below characterises the measurewise-amenable, \'etale groupoids for
which the ideal structure of $C^*(G)$ coincides with the $G$-invariant ideal structure of
$C_0(\go)$. The argument for the ``if" implication  is standard (see, for example,
\cite[Proposition~4.6]{Ren80}), but we include it for completeness.

The notion of amenability for groupoids is somewhat technical; for a detailed discussion, see
\cite{A-DR}.  For our purposes, we only need the following two facts.  First, if $G$ is measurewise
amenable, then $C^*(G) = C^*_r(G)$ \cite[Proposition~3.3.5]{A-DR}. Second, suppose that $U
\subseteq G^0$ is open and invariant.  If $G$ is measurewise amenable then each of $G_U$ and
$G_{\go\setminus U}$ is measurewise amenable \cite[Corollary~5.3.21]{A-DR}.\footnote{$G_U$ embeds
properly into $G$ since $G$ acts properly on itself.}

If $D \subseteq G^0$ is a closed invariant set, then $\{f \in C_c(G) : f|_{G_D}\equiv 0\}$ is an
ideal of $C^*(G)$ isomorphic to $C^*(G_{\go\setminus D})$, and the quotient is isomorphic to
$C^*(G_D)$ (see \cite[Lemma~2.10]{MRW96}). This decomposition fails in general for reduced
$C^*$-algebras.

\begin{cor}
\label{cor:maincor} Let $G$ be a second-countable, locally compact, Hausdorff groupoid. Suppose
that $G$ is measurewise amenable and \'etale. Then $D \mapsto \overline{\{f \in C_c(G) :
f|_{G_D} \equiv 0\}}$ is a bijection between closed invariant subsets of $\go$ and ideals of
$C^*(G)$ if and only if, for every closed invariant $D \subseteq \go$, $G_D$ is topologically
principal.
\end{cor}
\begin{proof}
First, we claim that there is a bijection between closed invariant subsets $D$ and ideals of the
form $I \cap C_0(\go)$, where $I$ is an ideal in $C^*(G)$. Let $D$ be a closed invariant subset.
Then the map that sends $D$ to the ideal $\{f \in C_0(\go) : f|_D \equiv 0\} \subseteq C_0(\go)$ is
a well defined injection.  To see that this map is a surjection onto the set of ideals of the form
$I \cap C_0(G)$, let  $I$ be an ideal of $C^*(G)$. Since the multiplication in $C_0(\go)$ is
pointwise, the ideal $I \cap C_0(\go)$ has the form $\{f \in C_0(\go) : f|_D \equiv 0\}$ for some
closed $D \subseteq \go$.  We show that $D$ is invariant by establishing that its complement is
invariant. Fix $\gamma \in G$ such that $s(\gamma) \not\in D$, and $f \in$ $I\cap C_0(\go)$ such
that $f(s(\gamma)) = 1$. We must show that $r(\gamma) \not\in D$. Let $B$ be an open bisection of
$G$ containing $\gamma$, and $h$ be a function supported on $B$ such that $h(\gamma) = 1$. By
Lemma~\ref{lem:technical}, $(h * f * h^*)(r(\gamma)) = |h(\gamma)|^2 f(s(\gamma)) = 1$, so
$r(\gamma) \not\in D$.  This proves our claim.

Now, it suffices to show that $I \mapsto I \cap C_0(\go)$ is a bijection if and only if $G_D$ is
topologically principal for each closed invariant $D \subseteq \go$.

First, suppose that $G_D$ is topologically principal for every closed invariant $D \subseteq \go$.
Fix an ideal $I$ of $C^*(G)$. Let $J$ be the ideal of $C^*(G)$ generated by $I \cap C_c(\go)$. Then
$J \subseteq I$. Since $G$ is measurewise amenable, $C^*(G) = C^*_r(G)$. Hence $J \not= \{0\}$ by
Proposition~\ref{it:Iso hypothesis}. We must show that $J = I$.

Let $J_0 := J \cap C_0(\go)$, and let $D := \{u \in \go : f(u) = 0\text{ for all } f \in J_0\}$; so
$J_0$ is the ideal $\{f \in C_0(\go) : f(u) = 0\text{ for all } u \in D\}$ of $C_0(\go)$.  As
above, $D$ is a closed invariant subset of $\go$. So \cite[Remark~4.10]{Renault:JOT91} implies that
restriction of functions induces an isomorphism $C^*(G)/J \cong C^*(G_D)$, and this isomorphism
carries $I/J$ to an ideal of $C^*(G_D)$ which has trivial intersection with $C_0(D)$ by
construction of $J$. Corollary~5.3.21 of \cite{A-DR} implies that $G_D$ is measurewise amenable, so
Proposition~\ref{it:Iso hypothesis} implies that $I/J$ is trivial and hence $I = J$ as required.

We prove the reverse implication by contrapositive. Suppose that there exists a closed
invariant subset $D$ of $\go$ such that  $G_D$ is not topologically principal.
Lemma~\ref{lem:equiv conditions} shows that $\Iso(G_D)^\circ \not= D$. Let $I(D)$ be the
ideal of $C^*(G)$ generated by $\{f \in C_c(\go) : f|_D \equiv 0\}$. Again by
\cite[Remark~4.10]{Renault:JOT91}, restriction of functions induces an isomorphism $\phi
: C^*(G)/I(D) \to C^*(G_D)$. Proposition~\ref{it:Iso hypothesis} applied to the groupoid
$G_D$ gives a nontrivial ideal $J$ of $C^*(G_D)$ such that $J \cap C_c(D) = \{0\}$. Let
$q_D : C^*(G) \to C^*(G)/I(D)$ and $q_J : C^*(G_D) \to C^*(G_D)/J$ be the quotient maps.
Let $K := \ker(q_J \circ \phi \circ q_D)$. That $J \cap C_c(D) = \{0\}$ forces $K \cap
C_c(\go)=C_0(D)$. That $J$ is nontrivial implies that $K \not= I(D)$. Since $K \cap
C_0(\go) = I(D) \cap C_0(\go)$, the result follows.
\end{proof}

\begin{rmk}
The hypothesis of measurewise amenability in Corollary~\ref{cor:maincor} is required
only to guarantee that $C^*(G_D) = C^*_r(G_D)$ for every closed invariant subset of $\go$. So the
theorem also holds under this formally weaker (but  less checkable)
hypothesis.
\end{rmk}

Recall that an \'etale groupoid $G$ is \emph{locally contracting} if for every nonempty open subset
$U$ of $\go$, there exists an open subset $V$ of $U$ and an open bisection $B$ such that
$\overline{V}\subseteq s(B)$ and $r(B\overline{V})\subsetneq V$ \cite[Definition~2.1]{AD97}. In the
following corollary, we use Theorem~\ref{thm:main} and Lemma~\ref{lem:equiv conditions} to
strengthen \cite[Proposition~2.4]{AD97}.

\begin{cor}
\label{cor: pure inf} Let $G$ be a second-countable, locally compact, Hausdorff groupoid.  Suppose
that $G$ is also locally contracting and  \'etale, and that $C^*(G)$ is
simple.  Then $C^*(G)$ is purely infinite.
\end{cor}
\begin{proof}
Theorem~\ref{thm:main} implies that $G$ is topologically principal, so
\cite[Proposition~2.4]{AD97} implies that every nonzero hereditary $^*$-subalgebra of $C^*(G)$
contains an infinite projection.
\end{proof}


\section{Examples} \label{sec:examples}

In this section, we present some examples to indicate why the hypotheses on our main
theorem are needed.  We also demonstrate that the final assertion of Lemma~\ref{lem:equiv
conditions} fails if $G$ is either not second-countable or not \'etale.

\begin{example}[Amenability]\label{ex:amenable}
Theorem~\ref{thm:main} cannot be strengthened to a characterisation of simplicity for
$C^*_r(G)$ for locally compact, Hausdorff, \'etale groupoids: the free group $\FF_2$ on
two generators, regarded as a discrete groupoid with just one unit, is a second
countable, locally compact, Hausdorff, \'etale groupoid that is not topologically
principal.  However, Powers proved in \cite{Powers:DMJ75} that $C^*_r(\FF_2)$ is simple.
\end{example}

\begin{example}[Twisted groupoid algebras]
Our characterisation of simplicity does not extend to groupoid $C^*$-algebras that are `twisted' by
a $2$-cocycle, as defined in \cite{Ren80}. To see why, consider the group $\ZZ^2$ regarded as a
discrete groupoid with one unit. This is a locally compact, Hausdorff, \'etale, amenable groupoid
with $\Iso(\ZZ^2) = \ZZ^2$, so our theorem reduces to the observation that $C^*(\ZZ^2) \cong
C(\TT^2)$ is not simple. To see that this does not extend to twisted algebras, fix $\theta \in
[0,1]\setminus \QQ$ and let $\phi_\theta : \ZZ^2 \times \ZZ^2 \to \TT$ be the $\TT$-valued
$2$-cocycle $\theta((m_1,m_2), (n_1,n_2)) = e^{i\theta(m_2 n_1)}$. It is well known that the
twisted groupoid $C^*$-algebra $C^*(\ZZ^2, \phi_\theta)$ is the irrational rotation algebra
$A_\theta$ and hence simple.
\end{example}

In Section~\ref{sec:Cstar} we were able to replace the hypothesis that $G$ is effective,
used in Section~\ref{sec:algebra}, with the more familiar hypothesis that it is
topologically principal.  The justification for this is
\cite[Proposition~3.6]{Renault:IMSB08}, which tells us that for second-countable,
Hausdorff and \'etale groupoids, the two hypotheses are equivalent. One might ask whether
the conditions are equivalent in general. The next example shows that for non-\'etale
$G$, effectiveness does not entail being topologically principal.

\begin{example}\label{ex:2ndctble ex}
Let $X := (0,1) \times \TT$. Define a continuous right action of $\RR$ on $X$ by
\[
(s, e^{i\theta})\cdot t = (s, e^{i(\theta + 2st\pi)}).
\]
Let $G$ be the transformation-group groupoid $X \rtimes \RR$. For each $u = (s,
e^{i\theta}) \in \go$, the isotropy group is $G^u_u = \{u\} \times \frac{1}{s}\ZZ$, so no
point in $\go$ has trivial isotropy. Fix an open set $U$ in $G$. We must show that $U
\setminus \Iso(G) \not= \emptyset$. Since $U$ is open, there exist $0 < a < b < 1$,
$\theta \in (0,2\pi)$, and $t \in \RR \setminus \{0\}$ such that $((a,b) \times
\{e^{i\theta}\}) \times \{t\} \subseteq U$. Fix $s \in (a,b)$. If $st \not\in \ZZ$ then
$((s, e^{i\theta}), t) \in U \setminus \Iso(G)$. So suppose that $st \in \ZZ$. Choose
$\varepsilon \in (0, \frac{1}{t})$ such that $s + \varepsilon \in (a,b)$. Then $st < (s +
\varepsilon)t < st+1$, so $(s + \varepsilon)t \not\in \ZZ$. Hence $((s +
\varepsilon,e^{i\theta}), t) \in U \setminus \Iso(G)$.
\end{example}

Our next example is also effective without being topologically principal. This time $G$
is \'etale and has totally disconnected unit space, but is not second-countable. This
shows that Lemma~\ref{lem:CKuniqueness} is strictly stronger than
\cite[Theorem~5.2]{CFST}.

\begin{example}\label{ex:etale ex}
Let $K$ denote the Cantor set and give $\TT$ the discrete topology. Let $X$ be the topological
product space $(K \cap (0,1)) \times \TT$. Define an (algebraic) action of $\RR$ on $X$ by
restriction of the action of Example~\ref{ex:2ndctble ex}. Endow the acting copy of $\RR$ with the
discrete topology. Then the action is continuous and the transformation groupoid $G$ is \'etale
(but not second-countable). Moreover, every open subset of $G$ which does not intersect $\go$
contains a subset of the form $((K \cap (a,b)) \times \{e^{i\theta}\}) \times \{t\}$ as in
Example~\ref{ex:2ndctble ex}, so arguing as in that example (using that $K \cap (a,b)$ has no
isolated points), we see that the interior of the isotropy subgroupoid is $\go$.
\end{example}

In Examples \ref{ex:2ndctble ex}~and~\ref{ex:etale ex}, $\go$ admits many nontrivial
closed proper invariant subsets. We do not have an example of a locally compact,
Hausdorff, \'etale, minimal, effective groupoid that is not topologically principal.


\section{Exel-Vershik systems}\label{sec:ExelVershik}

When we first began trying to prove that $\ag$ is simple if and only if $G$ is minimal and
topologically principal, we went looking for examples --- other than higher-rank graph groupoids --- of
\'etale groupoids with totally disconnected unit spaces to test the hypothesis. We were led to the
work of Exel and Vershik in \cite{EV06}. Their characterisation of simplicity
\cite[Theorem~11.2]{EV06} led us to condition~\eqref{it:3.1} of Lemma~\ref{lem:equiv conditions}
and from there to our main simplicity theorems. In this section, we investigate the relationship
between our result and that of Exel and Vershik.  We obtain a generalisation of their
simplicity theorem to a very broad class of dynamical systems.

Recall that an \emph{Ore semigroup} is a monoid $M$ which is cancellative and satisfies:
\begin{equation}
\label{eq: upward filtering}
\text{for all } m,n\in M, \text{ there exist } p,q\in M\text{ such that } pm=qn.
\end{equation}

\begin{dfn}
An $\emph{Exel-Vershik}$ system is a triple $(X,M,T)$ consisting of a second-countable, locally
compact, Hausdorff space $X$, a countable discrete Ore semigroup $M$, and an action $T$ of $M$ on
$X$ by local homeomorphisms; we write $T^m$ for the local homeomorphism associated to $m \in M$.
\end{dfn}

\begin{rmk}\label{rmk: com and group upward filtering}
Every commutative monoid and every group is an Ore semigroup: if $M$ is commutative then $nm=mn$
and if $M$ is a group then $m\inv m =  n\inv n$.
Indeed, a monoid $M$ is an Ore semigroup if and only if there is an embedding of $M$ in a group
$\Gamma = \Gamma(M)$ such that $\Gamma= M^{-1} M$ (see, for example, \cite[Theorem~1.2]{Lac00}).
The group $\Gamma$ is unique up to isomorphism and we call it the \emph{Grothendieck group of $M$}.
\end{rmk}

Let $(X,M,T)$ be an Exel-Vershik system and let $\Gamma(M)$ be the Grothendieck group of $M$.
Consider the set
\[
G(X,T):=\{(x,m\inv n,y)\in X\times \Gamma(M) \times X: m,n\in M, T^m(x)=T^n(y)\}.
\]
Remark~\ref{rmk: com and group upward filtering} and \cite[Proposition~3.1]{ER07} imply that the
formulas
\begin{align*}
r(x,m\inv n,y)&=x & s(x,m\inv n,y)&=y\\
(x,m\inv n,y)\cdot (y, p\inv q, z)&:=(x, m\inv np\inv q, z)&(x,m\inv n,y)\inv&=(y, n\inv m, x)
\end{align*}
make $G(X,T)$ into a groupoid.\footnote{Our convention for $G(X,T)$ is slightly different
than in \cite{ER07} for compatibility with Example~\ref{examples}\eqref{ex: dis group}.}

For precompact open subsets $U, V$ of $X$ and $m,n \in M$ such that $T^m|_U$ and $T^n|_V$ are
homeomorphisms and $T^m(U) = T^n(V)$, let
\[
    Z(U,V,m,n) := \{(x,m\inv n,y): x\in U, y\in V, T^m(x) = T^n(y)\}.
\]
Then the sets $Z(U,V,m,n)$ form a base of precompact open bisections for a second-countable
topology on $G(X,T)$. Under this topology, $G(X,T)$ is a locally compact, Hausdorff, \'etale
groupoid \cite[Proposition~3.2]{ER07}.

\begin{examples}
\label{examples}
\begin{enumerate}
\item \label{ex DR} If $M=\NN$ then  $G(X,T)$ is the Deaconu-Renault groupoid of the
    local homeomorphism $T$ \cite{Dea95}.
\item \label{ex: dis group} Let $M$ be a discrete group and suppose $T$ is an action of $M$ on
    $X$.  Then $T^g$ is a homeomorphism for all $g\in M$ so in particular a local
    homeomorphism. The Grothendieck group of $M$ is $M$. Further if $T^m(x)=T^n(y)$ then
    $x=T^{m\inv n} (y)$. So
    \[
        G(X,T)=\{(T^g(y), g, y): y\in X, g\in M\}
    \]
    and for each basic open set $Z(U,V,m, n)$, we have
    \begin{align*}
        Z(U,V,m,n)&=\{(T^{m^{-1}n}(y),{m^{-1}n},y):y\in V,\, T^{m^{-1}n}(y)\in U\}\\
            &=\{(T^{m^{-1}n}(y),{m^{-1}n},y): y\in V\cap T^{n^{-1}m}(U)\}.
    \end{align*}
    Thus the map $(T^g(y),g,y)\mapsto (y, g)$ induces an isomorphism of $G(X,T)$ with
    the transformation-group groupoid $X \rtimes_T M$.
\item Let $\Lambda$ be a row-finite higher-rank graph
    with no sources as in \cite{KumPas00}. Recall that $\Lambda^\infty$ denotes the
    infinite-path space of $\Lambda$ and that for $n \in \NN^k$ we write $\sigma^n$ for the
    shift map $\sigma^n(x)(p,q) = x(p+n, q+n)$ on $\Lambda^\infty$. The groupoid $\Gg_\Lambda$
    of \cite{KumPas00} is then identical to the groupoid corresponding to the Exel-Vershik
    system $(\Lambda^\infty, \NN^k, \sigma)$. Kumjian and Pask show that
    $\Gg_\Lambda$ is amenable in \cite[Theorem~5.5]{KumPas00}.
\end{enumerate}
\end{examples}

The next definition is an extrapolation of \cite[Definition~10.1]{EV06} to arbitrary
Exel-Vershik systems. This notion of topological freeness of $(X,M,T)$  is formally
weaker than that of \cite[Definition~1]{AS94} when $M$ is a countable discrete abelian
group.

\begin{dfn}
\label{def: top free} We say an Exel-Vershik system $(X,M,T)$ is \emph{topologically free} if for
every pair $m \neq n \in M$ the set $\{x \in X : T^m(x) = T^n(x)\}$ has empty interior.
\end{dfn}

\begin{prop}\label{prp:characterise top free}
An Exel-Vershik system $(X,M,T)$ is topologically free if and only if the associated groupoid
$G(X,T)$ is topologically principal.
\end{prop}
\begin{proof}
Suppose that $(X,M,T)$ is not topologically free. Then there exist $m\neq n\in M$ and an
open set $U\subseteq X$ such $T^m(x)=T^n(x)$ for all $x\in U$. Fix $z \in U$ and
neighbourhoods $W_m$ and $W_n$ of $z$ in $U$ such that  $T^l|_{W_l}$  is a homeomorphism
for $l=m,n$.  Define $V=W_m\cap W_n$. Then $z\in V$ and $T^l|_V$ is a homeomorphism for
both $l=m,n$. Since $T^m(x)=T^n(x)$ for all $x\in V\subseteq U$ we have $T^m(V)=T^n(V)$,
so the set $Z(V,V,m,n)$ is an open subset of $\Iso(G(X,T)) \setminus X$. Thus
Lemma~\ref{lem:equiv conditions} implies that $G(X,T)$ is not topologically principal.

Conversely, suppose that  $G$ is not topologically principal. By Lemma~\ref{lem:equiv
conditions}, there exists an open bisection $B\subseteq G(X,T)\setminus X$ such that
$r(\gamma)=s(\gamma)$ for all $\gamma\in B$.   So there is a basic open set $Z(U,V, m,n)$
contained in $B$. That $B \subseteq G(X,T)\setminus X$ forces $m\neq n$. Since $Z(U,V,
m,n) \subseteq B$ and $r(\gamma)=s(\gamma)$ for all $\gamma\in B$, we have $U=V$ and $T^m
x = T^n x$ for all $x \in U$. So $(X,M,T)$ is not topologically free.
\end{proof}

\begin{rmk}
The special case of Example~\ref{examples}(\ref{ex DR}) where $X$ is  a compact Hausdorff space and
$T:X\to X$  a covering map  was considered in \cite{CS09}.
Proposition~\ref{prp:characterise top free} implies that $(X,M,T)$ is topologically
principal, and so Proposition~\ref{it:Iso hypothesis} recovers \cite[Theorem~6
$((1)\Leftrightarrow (2))$]{CS09}.
\end{rmk}

\begin{rmk}
Recall from \cite[Definition~4.3]{KumPas00} that a row-finite higher-rank graph $\Lambda$ with no sources
is \emph{aperiodic} if for any $v\in \Lambda^0$ there exists $x\in \Lambda^\infty$ such that
$r(x)=v$ and $\sigma^n(x)\neq \sigma^m(x)$ for all $m\neq n\in \NN^k$. Recall also from
\cite[Definition~1]{RS07} that $\Lambda$ has \emph{no local periodicity} if for any $n\neq m\in
\NN^k$ and $v\in \Lambda^0$ there exists $x\in \Lambda^\infty$ such that $r(x)=v$ and
$\sigma^n(x)\neq \sigma^m(x)$. Kumjian and Pask show that $\Lambda$ is aperiodic if and only if
$G_\Lambda$ is topologically principal \cite[Proposition~4.5]{KumPas00}. A similar argument shows
that $\Lambda$ has no local periodicity if and only if the Exel-Vershik system $(\Lambda^\infty,
\NN^k, \sigma)$ is topologically free. Thus Proposition~\ref{prp:characterise top free} can be
viewed as a generalisation of \cite[Lemma~3.2]{RS07}.
\end{rmk}

Theorem~\ref{thm:main} and Proposition~\ref{prp:characterise top free} imply that if the full and
reduced $C^*$-algebras of the groupoid $G(X,T)$ of an Exel-Vershik system $(X,M,T)$ coincide, then
the associated $C^*$-algebra $C(X) \rtimes_T \Gamma(M)$ is simple if and only if the system is
topologically free and for each $x \in X$ the orbit
\[
    [x]_T := \{y \in X : T^m y = T^n x \text{ for some }m,n \in M\}
\]
is dense in $X$. It is therefore an interesting question whether $C^*(G(X,T)) = C^*_r(G(X,T))$
whenever $\Gamma(M)$ is amenable. We give a partial answer which applies to all systems for which
Exel and Renault's results guarantee that the Exel crossed product $C(X) \rtimes_T \Gamma(M)$ of
\cite{Exel:EDTS08} coincides with $C^*(G(X,T))$.

\begin{cor}
Suppose $M$ is an Ore semigroup  such that $\Gamma(M)$ is amenable. Suppose that $(X, M, T)$ is an
Exel-Vershik system satisfying the standing hypotheses~4.1 of \cite{ER07}. Then $C^*(G(X,T)) =
C^*_r(G(X,T))$. Moreover, $C(X) \rtimes_T \Gamma(M)$ is simple if and only if the system is
topologically free and $\overline{[x]}_T = X$ for each $x \in X$.
\end{cor}
\begin{proof}
The second assertion follows from Theorem~\ref{thm:main} once we show that $C^*(G(X,T)) =
C^*_r(G(X,T))$. For this let $\pi$ be the isomorphism $\pi : C(X) \rtimes_T \Gamma(M) \cong
C^*(G(X,T))$ of \cite[Theorem~6.6]{ER07}, and let $q : C^*(G(X,T)) \to C^*_r(G(X,T))$ be the
quotient map. It suffices to show that $q \circ \pi$ is injective. For this, just run the proof of
\cite[Theorem~6.6]{ER07} replacing $C^*(G(X,T))$ with $C^*_r(G(X,T))$. It is only necessary to
check that $q \circ (\pi \times \sigma)$ is injective on each graded subspace, and for this the
argument of \cite[Proposition~6.5]{ER07} suffices because the calculations in that proof involve
elements of $C_c(G(X, T))$.
\end{proof}

Amenability is irrelevant to the Steinberg algebras of Section~\ref{thm:alg_simple}. So
Exel-Vershik systems $(X, M, T)$ where $X$ is a Cantor set should provide interesting
examples of Steinberg algebras $A(G(X,T))$ for which simplicity is characterised by
Theorem~\ref{thm:alg_simple}.



\end{document}